\documentclass[12pt]{article}
\usepackage[utf8]{inputenc}
\usepackage{amsmath,amssymb}
\usepackage{times}
\usepackage{array}
\usepackage[english]{babel}
\usepackage{wrapfig} 
\usepackage{amssymb,amsthm,color}
\usepackage[active]{srcltx}
\input epsf
\usepackage{graphicx}
\usepackage{hyperref,pdfsync}
\usepackage{changes}
\definecolor{wineRed}{rgb}{0.7,0,0.3}

\theoremstyle{definition}
\newtheorem{theorem}{Theorem}[section]
\newtheorem{defn}[theorem]{Definition}
\newtheorem{lemma}[theorem]{Lemma}
\newtheorem{proposition}[theorem]{Proposition}
\newtheorem{cor}[theorem]{Corollary}

\newtheorem{remark}[theorem]{Remark}

\def\conv{{\hbox{conv}\,}}

\def\diam{\hbox{diam}\,}
\def\dist{\hbox{dist}\,}

\def\supp{\hbox{supp}\,}
\def\Int{\hbox{int}\,}
\def\bN{\mathbb{N}}
\def\bR{\mathbb{R}}

\def\bX{\mathbb{X}}
\def\bV{\mathbb{V}}

\def\d{\partial}\def\p{\partial}

\def\ep{\epsilon}

\def\cH{{\mathcal{H}}}
\def\cI{{\mathcal{I}}}
\def\cK{{\mathcal{K}}}

\def\cN{{\mathcal{N}}}

\numberwithin{equation}{section}
\newcommand{\subjclass}[1]{\bigskip\noindent\emph{2010 Mathematics Subject Classification:}\enspace#1}
\newcommand{\keywords}[1]{\noindent\emph{Keywords:}\enspace#1}

\begin{document}
\title{The planar Least Gradient problem in convex domains: the discontinuous case}
\author{Piotr Rybka\footnote{Faculty of Mathematics, Informatics and  Mechanics, the University of Warsaw,
ul. Banacha 2, 02-097 Warsaw, Poland, email: rybka@mimuw.edu.pl}, Ahmad Sabra\footnote{Department of Mathematics, American University of Beirut, Riad el solh, 1107 2020 Beirut, Lebanon. email: asabra@aub.edu.lb}}
\date{}

\maketitle
\begin{abstract}
We study the two dimensional least gradient problem in  convex polygonal sets in the plane, $\Omega$. We show  the existence of solutions when the boundary data $f$ are attained in the trace sense. The main difficulty here is  a possible discontinuity of $f$. Moreover, due to the lack of strict convexity of $\Omega$, the classical results are not applicable. 
We state the admissibility conditions on the boundary datum  $f$, 
that are sufficient for establishing an existence result. One of them is that $f\in BV(\d\Omega)$. The solutions are constructed  by a limiting process, which uses solutions to known problems.


 \subjclass{Primary: 49J10, Secondary: 49J52, 49Q10, 49Q20}

 \keywords{least gradient, trace solutions, convex  but not strictly convex domains, $BV$ functions.}
 
\end{abstract}

\section{Introduction}
The least gradient problem, in its isotropic and anisotropic versions, has recently attracted considerable attention, see \cite{Filippo}, \cite{gorny2}, \cite{jerrard}, \cite{korte}, \cite{zuniga}.
It may be stated as follows: for a given bounded region $\Omega\subset \bR^N$ with Lipschitz continuous boundary and a boundary datum $f\in L^1(\d\Omega)$, we seek $u\in BV(\Omega)$ which  attains the following  minimum,
\begin{equation}\label{lg}
 \min \left\{ \int_\Omega |D u|: \ u\in BV(\Omega),\ T u = f\right\},
\end{equation}
where 
$T: BV(\Omega) \to L^1(\partial\Omega)$ denotes the trace operator. 
There are various  motivations for studying (\ref{lg}). In one of the early papers,  the authors of \cite{parks} were interested in solutions to (\ref{lg}), because their level sets are minimal surfaces. A plastic design problem leads to an obstacle  least gradient problem, \cite{kohn}, where the solution $u$ has to satisfy a pointwise constraint, $u(x)\le \phi(x)$, where $\phi$ is given. 

More recently, a link between (\ref{lg}) and the Free Material Design or the minimal flow Beckman problem, see \cite{grs}, permitted to apply tools of the optimal transportation theory to the study of the least gradient problems as shown in \cite{Filippo} and \cite{dweik-gorny}.

%

A weighted least gradient problem appears in medical imaging, \cite{nachman}, and \cite{timonov}, which requires investigating the anisotropic version of \eqref{lg}, see \cite{jerrard}.  Namely, authors of these papers consider (\ref{lg}) with the integrand $|Du|$ replaced by $a(x) |Du|$, where $a$ is subject to some restrictions.

One of the first papers devoted to systematic studies of (\ref{lg}) was the article by Sternberg-Williams-Ziemer, \cite{sternberg}. It offered a geometric construction of solutions (\ref{lg}), because the direct method of the calculus of variations is not available due to the well-known lack of lower semicontinuity of the total variation over the set $\{u\in BV(\Omega): Tu = f\}$ for $f\in L^1(\d\Omega)$. The construction was performed for a restricted class of regions and continuous data. Namely, it is assumed in \cite{sternberg} that the boundary of the region $\Omega$ has a non-negative mean curvature (in a weak sense) and $\d\Omega$ is not locally area minimizing. In the anisotropic case studied in \cite{jerrard} and  \cite{moradifam},  $\Omega$ was supposed to satisfy a barrier condition that is equivalent to the conditions in \cite{sternberg} for the isotropic setting. However, in the case of regions in  $\bR^2$ these conditions 
reduce to the strict convexity of $\Omega$. We relax the strict convexity here. 
We will consider only polygonal regions  in this paper, as we did in \cite{RySa}. In other words, $\d\Omega$ consists of at most countable number of line segments, 

Sternberg-Williams-Ziemer also showed uniqueness of the solution to (\ref{lg}) provided that $f\in C(\d\Omega)$, see \cite{sternberg}. Using a weaker interpretation of the boundary conditions, the authors of \cite{mazon} proved the existence of solutions to a relaxed least gradient problem for general Lipschitz domain with $L^1$ boundary data, see \cite[Definition 2.3]{mazon}. Moreover, the example in \cite{mazon} shows that even a finite number of discontinuity points leads to the loss of  uniqueness of solutions.

We mention that the non-uniqueness could be tackled. The author of \cite{gorny2} provides a classification of multiple solutions. It is worth noticing that this result is valid for convex regions, which  need not be strictly convex.

It is interesting to ask if we can relax the continuity of the data in the existence theorems. Examples show, see \cite{spradlin}, that the space of traces of solutions to (\ref{lg}) is smaller than $L^1(\d\Omega)$, which the image of the trace operator $T$. It is known that $f\in BV(\d\Omega)$ or a.e. continuity  of $f$ suffices for the existence of solutions, at least in case of strictly convex $\Omega$, see \cite{Filippo}, \cite{gorny3}.

Here our objective comes. We want to study (\ref{lg}) in convex polygonal domains in the plain with data in $BV(\d\Omega)$.

In general, we do not expect 
existence of solutions to \eqref{lg}, even in the case of continuous $f$, once the strict convexity condition is dropped, see \cite[Theorem 3.8]{sternberg}. As a result,  we
have to develop a proper tool to examine the domain and the range of the data. Actually, we did it in \cite{RySa}, when the datum $f$ was continuous.  
In order to avoid unnecessary technical difficulties, we restricted our attention to convex bounded polygons $\Omega$ that have finite or infinite number of sides. 

In \cite{RySa} we  stated admissibility conditions. The first one, C1, means that $f$ restricted to a side is a monotone function. The other condition, C2, says that if $f$ achieves a maximum or minimum on a side of $\d\Omega$, then this must happen on a large set called a hump. More precisely, by a {\it hump} we mean an interval on which $f$ attains a local maximum or minimum, see Section \ref{adm} for details. 
Moreover, each hump must
have a ``companion" on a different side of $\d\Omega$. In this way we avoid accumulation of the level sets of solutions along a side, which is a common cause of nonexistence of solutions. 

Condition C2 had to be complemented with another restriction on $f$ ruling out
a category of bad data. For this reason, we  introduced in \cite{RySa} the Ordering Preservation Condition (OPC for short), see Definition \ref{dfopc}, and the Data Consistency Condition (DCC for short), see   Definition  \ref{dfdcc}.
Roughly speaking,
the OPC condition does not permit datum $f$, which leads to intersections of the level sets of the candidates for solutions. At the same time, the DCC 
says that, if $f$ attains a maximum on side $\ell_1$, then $f$ may  not  attain any minimum on side $\ell_2$ in front of  $\ell_1$. 

We presented in \cite{RySa} examples of data  showing  that dropping C1, C2, OPC or DCC leads to non-existence of solutions. 


Here, we have to state these restrictions in a way suitable for discontinuous data in $BV(\d\Omega)$, they will be called D1, see Definition \ref{admDC1} and D2, see Definition \ref{admDC2}. 
We also add one more admissibility condition D3, see Definition \ref{admDC3}, which prevents jumps at points which are strict local maxima or minima.



We may now state the results of this paper:

\begin{theorem}\label{tsuf-C} 
Let us suppose that $f\in BV(\d\Omega)$, choosing a representative that satisfies  (\ref{rd0}), 
$\Omega$ is an open, bounded and convex set.  The boundary of $\Omega$ is polygonal, i.e. it consists of at most a countable number of sides,  $\d \Omega = \bigcup_{j\in \mathcal{J}} \ell_j$, where $\ell_j$ are line segments. Furthermore,  $f$ satisfies the
admissibility conditions D1 or  D2 on all sides of $\d\Omega$,  and D3, the Ordering Preservation Condition (\ref{OPC}) and the Data Consistency Condition (\ref{dcc1}-\ref{dcc2}).\\
(a) If the number of sides  as well as the number of humps are finite, then there exists a solution to  problem (\ref{lg}). Moreover, the following estimate is valid for any solution $u$ of (\ref{lg}), 
\begin{equation}\label{rm-main}
\| u \|_{L^\infty} \le \| f\|_{L^\infty},\qquad
\int_\Omega |Du | \le \diam \Omega \, TV (f).
\end{equation}
(b) We assume that the number of sides is infinite and  they have
at most one accumulation point $p_0$. We also require  that the number of all humps of $f$ is finite
and $f$ has 
a local maximum/minimum at $p_0$. Moreover, there is $\epsilon>0$ such that the restriction of $f$ to each component of $(B(p_0,\epsilon)\cap\d\Omega)\setminus \{p_0\}$ is strictly monotone.
Under these assumptions the problem (\ref{lg}) has a solution, which satisfies (\ref{rm-main}).\\
(c) Suppose that only one side, $\ell$, has an infinite number of humps accumulating at its endpoint $p_0$. This point $p_0$ may be an accumulation point of sides of $\Omega$. 
Then, problem (\ref{lg}) has a solution, satisfying (\ref{rm-main}).
\end{theorem}



Let us make a few comments. We remark that estimate (\ref{rm-main}) is new also in the case of continuous data. However, its proof requires the restricted geometry of the data. 

Our method of proof is based on  successive  approximations. 
In fact, we approximate a given datum $f$ from below  and above by monotone sequences satisfying the assumptions of the existence result  \cite[Theorem 3.8]{RySa}. The monotonicity of the sequences implies their pointwise convergence. In a sense we construct an upper and a lower solution in part (a). The monotonicity is very helpful in establishing (\ref{rm-main}).

In the proof of part (b) we approximate not only the data but also the region $\Omega$. We do so in such a way that part (a) yields a sequence of approximate solutions. Estimate  (\ref{rm-main}) is very helpful in establishing convergence and the trace property, i.e. $Tu =f$. The approach to prove (c) is the same as in the case of (b), however, the details are a bit different.

Finally, we present the plan of the paper. In Section \ref{adm} we recall the admissibility conditions stated in \cite{RySa} and used to prove the existence results. We adapt these conditions to the case of discontinuous data. We also prove a comparison principle  and the estimate (\ref{rm-main}) for continuous data. Section \ref{sDdata} is devoted to the construction of solution by approximation processes. Here, we  carefully construct decreasing and increasing sequences of continuous functions on $\d\Omega$, which approximate the datum. The point is that these sequences satisfy the assumptions of the basic existence result in \cite[Theorem 3.8]{RySa}, when $\Omega$ is a polygon and $f$ has a finite number of humps, i.e. flat pieces of the graph of $f$. The existence of solutions in this case is shown in Theorem \ref{Main2}. These solutions are used to approximate situations when $\d\Omega$ has an infinite number of sides, but a finite number of humps, cf. Theorem \ref{td-nsk} or when the number of humps is infinite, see Theorem \ref{Main3}.

\section{Preliminary results}\label{adm}

In this paper we assume that $\Omega\subseteq \mathbb R^2$ is a bounded convex domain, whose boundary is the union of at most countable number of line segments. We denote the sides of $\partial \Omega$ by $\ell_j=[p^j_l,p^j_r]$ and write $$\partial \Omega=\bigcup_{j\in\mathcal J} \ell_j.$$

Given a function $f\in BV(\partial \Omega)$, we will always choose a representative of $f$ whose jumps are a bit restricted. In case of functions defined on the real line we would consider the so-called good representatives of $f$. Namely, we require that the following condition holds,
\begin{equation}\label{rd0}
    f(x) \in [\liminf_{z\to x} f(z),\limsup_{z\to x}f(z)]\qquad \hbox{for all }x\in \d\Omega.
\end{equation}
This will be our standing assumption. 
Having the above setting in mind, we define the following:

\begin{defn}\label{image set}
Given $f\in BV(\d\Omega)$, satisfying (\ref{rd0}),
we define
$$P(e)=\left\{x\in \partial \Omega: e\in [\liminf_{z\to x} f(z),\limsup_{z\to x}f(z)]\right\}.$$
We notice that if $f$ is continuous in $\d\Omega$, then $P(e)=f^{-1}(e)$.
\end{defn}

\begin{proposition}\label{prop:level set}
For $f\in BV(\partial\Omega)$, satisfying (\ref{rd0}), 
the set $P(e)$ is closed.
\end{proposition}
\begin{proof}
We shall show that the complement of $P(e)$ is open. Let $x\notin P(e)$, that is $$
e\notin [\liminf_{z\to x} f(z), \limsup_{z\to x} f(z)].
$$ 
We will consider the case, when 
$$e<\liminf_{z\to x} f(z)=\sup_{\varepsilon>0}\left(\inf \{f(z):z\in (B(x,\varepsilon)\setminus \{x\})\cap \d\Omega\right\}).$$
The argument, when $e> \limsup_{z\to x} f(z)$ is similar.

By definition of $\liminf$, for a given $\delta>0$ there exists  $\varepsilon_0$ such that for every $z\in B(x,\varepsilon_0)$, $e<f(z)-\delta$.  Hence, $$e\leq \liminf_{z\to y} f(z)-\delta<\liminf_{z\to y} f(z)$$ for every $y\in B(x,\varepsilon_0)\cap \partial \Omega$, i.e. $B(x,\varepsilon_0)\cap P(e)=\emptyset$. We conclude that the complement of $P(e)$ is open.
\end{proof}
\begin{defn}\label{def:Hump}
For a given  $f\in BV(\partial \Omega)$, with each side $\ell$,
we associate  
a family of closed intervals $\{\bar I_i\}_{i\in \cI}$ such that $
\bar I_i =[a_i, b_i]\subsetneq \ell$ and $\bar I_i\cap \d\ell =\emptyset$. We assume that on each $I_i$, the function  $f$ is constant and attains a local maximum or minimum, and each  $I_i$ is maximal with this property. 
We call such $I_i$ a {\it hump}, and set $e_i = f(I_i)$, $i\in\cI$. We also adopt the following notation. Given a side $\ell=[p_l,p_r]$ in $\partial \Omega$,  and a hump $I$, $\bar I=[a,b]\subseteq \ell$, we order the vertices so that 
$$
|p_l - a|<|p_l - b|\quad\hbox{and}\quad|p_r - b|<|p_r - a|.
$$ 
In other words, by definition the endpoint $a$ of $I$ is closer to the `left' endpoint of side $\ell$. Of course, the last notion depends on the orientation on $\d\Omega$. 
To be precise, we require that the tangent vector $p_r - p_l$ agrees with the positive orientation of $\d\Omega.$
\end{defn}
Once for all we fix the positive orientation of $\Omega.$ After doing so, we may identify a connected neighborhood of any point $x_0\in\d\Omega$ with an open interval of the real line by an order preserving homeomorphism. Keeping this in mind we may sensibly write
$$
\lim_{x\to x_0^-} f(x) =: f(x_0^-),\qquad \lim_{x\to x_0^+} f(x) =: f(x_0^+).
$$

\subsection{Admissibility conditions}\label{subsec:adm condition}
We are now ready to present the conditions on $f$ and $\Omega$ that are sufficient for a proof of existence of 
solutions to the corresponding least gradient problem \eqref{lg}. Since, in this paper, the trace $f$ is not necessarily continuous, the admissibility conditions in \cite{RySa} need to be adjusted. 

\begin{defn}\label{admDC1} Let us suppose that $f:\d\Omega\to \bR$ is a function of bounded variation and $f$ restricted to  $\ell =[p_l,p_r]$ is
monotone. We shall say that $f$ satisfies the {\it admissibility condition} D1 if 
one of the following conditions holds:\\
(i)  $f$ is continuous at both endpoints of $\partial\ell$;\\
(ii) there is $\ep>0$ such that $f$ restricted to $\ell \cup (B(p_t,\ep)\cap \d\Omega)$ is monotone, where $t=l$, or $r$, and $f$  is continuous at $\d\ell\setminus \{p_t\}$;\\
(iii) there is $\ep>0$ such that $f$ restricted to 
$\ell \cup (\d\Omega \cap (B(p_l,\ep)  \cup B(p_r,\ep))$
is monotone. 
\end{defn}

In case $f$ has a hump on  $\ell$, a side of $\d\Omega$, we require the following D2 condition.
\begin{defn}\label{admDC2}
We say that  function $f\in BV(\d\Omega)$ satisfies the {\it admissibility condition} D2 on  side $\ell$ if and only if the following restrictions are in force:\\
(a) For each hump $I_i$ in $\ell$, with $\bar I_i=[a_i, b_i]$, the following inequality holds,
 \begin{equation}\label{DA}
 \dist(a_i, P(e_i)\cap(\d\Omega\setminus \bar I_i)) + \dist(b_i, P(e_i)\cap(\d\Omega\setminus \bar I_i)) < |a_i-b_i|.
\end{equation}
Due to Proposition \ref{prop:level set} there are points at which the distances above are attained. We then require in addition that if $y_i$, $z_i\in\d\Omega$ are such that 
\begin{equation}\label{ddfyz}
\dist(a_i, P(e_i)\cap(\partial\Omega\setminus \bar I_i)) = 
\dist(a_i, y_i),\quad 
\dist(b_i, P(e_i)\cap(\partial\Omega\setminus \bar I_i)) = \dist(b_i, z_i),
\end{equation}
then $y_i$, $z_i\in\d\Omega\setminus\ell.$\\
(b) If $J$ is  one of the components of $\ell \setminus \bigcup_i I_i$ such that $\partial J \cap \partial \ell \neq\emptyset$, then $f|_J$ satisfies D1 with $J$ in place of $\ell.$
\end{defn}

The admissibility conditions D1 and D2 deal with the behavior  of discontinuities  of $f$ 
at the vertices of 
$\partial \Omega$ and the hump endpoints. We have to restrict the behavior of $f$ at other discontinuity points. 
The following condition D3 serves this purpose. 

\begin{defn}\label{admDC3}
We say that $f\in BV(\d\Omega)$ satisfies the admissibility condition D3 on a side $\ell$ if and only if for every discontinuity point $x_0$ in the interior of $\ell$, there exists $\varepsilon>0$ such that $f$ restricted to $B(x_0,\varepsilon)\cap \ell$ is monotone.
\end{defn}
\begin{remark}\label{mono}
 We notice that D3 implies that if $J$ is a line segment contained in the complement of all humps, then $f$ must be monotone on $J$.
\end{remark}

\begin{remark}\label{rmk:C1,C2}
We notice that the cases when $f$ is continuous, the admissibility conditions D1 and D2 reduce to the admissibility conditions for continuous boundary conditions discussed in \cite{RySa}. Indeed, if 
$f$ is continuous on $\partial \Omega$ and satisfies conditions D1 on a side $\ell$, then $f$ is  monotone on $\ell$. This corresponds to condition C1 in \cite[Definitions 2.1]{RySa}. 
Notice also, that if $f$ is continuous, then for every $e\in f(\p\Omega)$, $P(e)=f^{-1}(e)$ and hence condition D2  corresponds to condition C2 in \cite[Definition 2.2]{RySa}.  Moreover, D3 is satisfied automatically.
\end{remark}

As in the continuous case, we require the following compatibility conditions OPC, and DCC otherwise solution to \eqref{lg} might fail to exist, 
see \cite{RySa}.

\begin{defn}\label{dfopc}
We shall say that $f\in BV(\d\Omega)$ satisfies the 
{\it order preserving condition}, (OPC),
if for any two different humps $I_1$, $I_2$, contained in two sides $\ell_1$, $\ell_2$, which may be equal, 
any choice of the corresponding 
points $y_i, z_i$, $i=1,2$ defined in (\ref{ddfyz}) fulfills
\begin{equation}\label{OPC}
 ([a_1,y_1] \cup [b_1, z_1]) \cap ([a_2,y_2] \cup [b_2, z_2]) = \emptyset.
\end{equation}
\end{defn}
The Order Preserving Condition rules out nonsense data, because level sets of solutions cannot cross, but by itself it is not sufficient. This is why we introduced another requirement, complementing (\ref{OPC}). In order to do so, we present a new piece of notation. 
\begin{defn}\label{darc}
For a given  hump $I$, with $\bar I=[a,b]$, and corresponding points $y, z$, see (\ref{ddfyz}), we let $\overline{yz}_{ab}\subset\partial\Omega$ to be the polygonal-arc connecting $y$ and $z$, and
not containing $[a,b]$.
\end{defn}
 For any $p\in \d\Omega$, $\epsilon>0$   we write,
$$
\cN(p,\epsilon)=B(p,\epsilon)\cap\partial\Omega\setminus\overline{yz}_{ab} .
$$

\begin{defn}\label{dfdcc}
We shall say that $f\in BV(\partial\Omega)$ satisfies the 
{\it data consistency condition}, (DCC for short), provided that  at all humps $I$, with $\bar I=[a,b]$, 
 there is a choice of points $y$, $z$ such that 
\begin{equation}\label{dcc1}
\inf_{x\in \overline{yz}_{ab}} f(x) \ge f(I),
\end{equation}
whenever $f$ attains a local maximum on hump $I$. Here, the points $y, z$ are defined in (\ref{ddfyz}) and  the arc
$\overline{yz}_{ab}\subset \p\Omega$ is defined above.
Moreover,  there is $\epsilon>0$ such that
\begin{equation}\label{dcc2}
\begin{array}{ll}
f(p_2) < f(p_1)& \hbox{if } p_1, p_2 \in \cN(z,\epsilon),\ \dist(p_2,z)<\dist(p_1,z),\\
f(p_2) < f(p_1)& \hbox{if }
p_1, p_2 \in \cN(y,\epsilon),\ \dist(p_2,y)<\dist(p_1,y).
\end{array}
\end{equation}
If $f$ attain a local minimum on hump $I$, then the inequalities in (\ref{dcc1}) and (\ref{dcc2}) are reversed.
\end{defn}

\subsection{Auxiliary results for solutions with continuous data}

It is shown in \cite[Theorem 3.8]{RySa}, that if $\partial \Omega$ has finitely many sides and if $f\in C(\partial \Omega)$ satisfies conditions C1, C2 on all sides of $\d\Omega$ as well as the OPC and DCC conditions and $f$ has finitely many humps, then a unique solution to \eqref{lg} exists.

The solution to the continuous case is constructed as follows. By \cite[Lemma 3.1]{RySa}, there is  a decreasing sequence of strictly convex domains $\Omega_n\subset \mathbb R^2$ and such that $\bar\Omega_n$ converges to $\bar\Omega$ in the Hausdorff measure. Moreover, 
the vertices of $\Omega$ belong to the boundary of $\Omega_n$.

Let $\pi: \mathbb R^2\mapsto \bar \Omega$ be the orthogonal projection onto the closed convex set $\bar\Omega$ and $\pi_n$ be its restriction to $\d \Omega_n$, and $f_n=f\circ\pi_n$. Due to \cite{sternberg}, there exists a unique solution $v_n$ to the least gradient problem \eqref{lg} on $\Omega_n$ with trace $f_n$. Moreover, the solutions $v_n$ are equicontinuous with modulus of continuity that depends only on $f$ and the number of sides of $\d\Omega$, see \cite[Lemma 3.6]{RySa}. After letting $u_n=v_n|_{\d\Omega}$, we showed that $u_n$ converges uniformly to a least gradient function $u$, the unique solution to \eqref{lg}.

Having in mind the above construction, we show the following comparison principle for solutions to the least gradient solution to \eqref{lg} in the case where $\Omega$ has finitely many sides, and $f\in C(\d\Omega)$.

\begin{proposition}\label{comparison}
Let $\Omega$ be convex and $\d\Omega$ is a polygon with finitely many sides. We assume that $f_1$ and $f_2$ belong to  $C(\partial \Omega)$ satisfying the admissibility conditions C1, C2, as well as the OPC, and DCC and such that $f_1\leq f_2$. Let $u_1$ and $u_2$ be the corresponding unique solutions to \eqref{lg}, 
then $u_1\leq u_2$ in $\Omega$.
\end{proposition}

\begin{proof}
Using the notation in the construction above, we let $\Omega_n$ be the strictly convex sets converging to $\Omega$.
We assume that $v_1^n$ (resp. $v_2^n$) are the unique solution to \eqref{lg} on $\Omega_n$ with the corresponding trace $f_1^n=f\circ \pi_n$ (resp. $f_2^n=f\circ \pi_n$) and $u_1^n$ (resp. $u_2^n$) its restriction to $\Omega$. By definition, we have $f_1^n\leq f_2^n$, then by \cite{sternberg}, we know that $v_1^n\leq v_2^n$.
Then since $u_1^n=v_1^n|_{\Omega}$ and $u_2^n=v_2^n|_{\Omega}$ converge correspondingly to $u_1$ and $u_2$ in $\Omega$, we conclude that
 $u_1\leq u_2$.
 \end{proof}


We close this section with important estimate on the solution, which we will use later in its full power. Its  weaker version appeared in the  course of the proof of  \cite[Theorem 3.10]{RySa}.
\begin{lemma} \label{DuleTV}
 Let us suppose that $f\in C(\d\Omega)\cap BV(\d\Omega)$ and the conditions of the existence Theorem \cite[Theorem 3.8]{RySa} are satisfied, i.e.
 $\Omega$ is an open, bounded and convex set, whose boundary is a polygon and $\{\ell_j\}_{j\in \cI}$ is the finite family of sides of $\d\Omega$. 
In addition we assume that the number of humps is finite, $f$ satisfies the
admissibility conditions C1 or C2 on all sides
of $\d\Omega$, as well as the complementing ordering preservation condition, (\ref{OPC}), and the data consistency condition, (\ref{dcc1}-\ref{dcc2}).
If  $u$ is the solution to the corresponding least gradient problem (\ref{lg}), then
 $$
 |Du|(\Omega) \le \left(\diam \Omega\right) TV(f).
 $$
\end{lemma}
\begin{proof}
Let $\bar I_i=[a_i,b_i], i\in \mathcal I$, be the closure of the humps of $f$, and $\bV =\{v_1,\ldots,v_K\}$ be the vertices of $\d \Omega$. We denote by $\mathbb X=\bV \cup \left(\cup_{i\in I}\{a_i,b_i\}\right)$ the set whose only elements are the hump endpoints and vertices of $\d\Omega$.

We notice that 
$$
\d \Omega \setminus\left(\bigcup_{i\in \mathcal{I}} I_i \cup \bV\right)
= \bigcup_{j=1}^M J_j,
$$
where $J_j$ are line segments contained in the sides of $\d\Omega$. The assumptions we made on data imply that $f$ restricted to each of the intervals $J_j$, $j=1,\ldots, M$ is monotone. 

{\it Step 1.} Now, we will assume that $f$ on each $J_j$ is strictly monotone. We will denote by $\Int I_i$, $i\in \cI$, the relative interior of the humps. Then, due to strict monotonicity of $f$ we have,
$$
\left(\d\Omega \setminus \cup_{i\in\cI} \Int I_i \right)\cap f^{-1}(f(\bX))
$$
consists of a finite set of points, $ \{\xi_1,\ldots, \xi_N\}=:\Xi$. Moreover, we arrange the set 
$f(\bX)= \{y_1,\ldots, y_S\}$, so that
$$
\min f= y_1< y_2< \ldots < y_S = \max f.
$$
The assumption that $f$ is continuous and strictly monotone on each $J_j$ facilitates computation of the total variation of $Du$. Due to the coarea formula, we have
$$
\int_\Omega |D u| = \int_{\min f}^{\max f} P(\{ u \ge t\}, \Omega)\,dt= \sum_{k=2}^S\int_{y_{k-1}}^{y_k}P(\{ u \ge t\}, \Omega)\,dt.
$$

Obviously, the length of each of the components of $\d\{ u \ge t\}$ does not exceed $\diam \Omega$. We have to count the components of $\d\{ u \ge t\}$, $t\in [y_{k-1}, y_k),$ their number will be written as $K(t)$. We denote by $(\xi^k_i, \xi^{k-1}_j)\subset \bR^2$ an open line segment with endpoints $\xi^k_i$ and $\xi^{k-1}_j$. We can find  such intervals satisfying the following conditions,
\begin{equation}\label{rxi}
    \xi^k_i, \xi^{k-1}_j\in \Xi,\quad f( \xi^{k-1}_j)= y_{k-1}, \ f(\xi^k_i) = y_k,
\quad  (\xi^k_i, \xi^{k-1}_j) \subset \d\Omega  \quad (\xi^k_i, \xi^{k-1}_j) \cap \Xi =\emptyset.
\end{equation}
Roughly speaking, the last two conditions mean that  $\xi^k_i$ and $\xi^{k-1}_j$ are the closest neighbors in $\Xi$.

Let us take $J$ a component of $\d\{ u \ge t\}$ such that $J\cap [\xi^k_i, \xi^{k-1}_j] \neq \emptyset.$ Actually, for each $J$ we can find two intervals $[\xi^k_i, \xi^{k-1}_j]$, $[\xi^k_l, \xi^{k-1}_m]$ with the above property. Thus,
$$
K(t) \le \# \cK,
$$
where
$$
\cK =\{ \{\xi_i^{k-1},\ \xi_j^k\}\subset \Xi: (\ref{rxi})\hbox{ holds}\}.
$$
This implies the following estimate
$$
\int_{y_{k-1}}^{y_k} P(\{ u\ge t\}, \Omega) \,dt
\le K(t) \diam\Omega\, (y_k- y_{k-1}) =  \diam\Omega
 \sum_{ \{\xi^{k-1}_i,\ \xi^k_j\}\in \cK} 
f(\xi^k_i) - f(\xi^{k-1}_j).
$$
Since the points of $\Xi$ used in the sum above form a partition of $\d\Omega,$
we deduce that
$$
\int_\Omega |D u| \le \diam \Omega\, TV(f). 
$$
{\it Step 2.} Now, we relax the assumption of strict monotonicity of $f$ on $J_j$. It is easy to construct  $\varphi^j_n\in C^\infty(J_j)$ such that:\\
1) each $\varphi^j_n$ is strictly monotone on $J_j$;\\
2) $(\varphi^j_n - f)|_{\d I_j} =0$;\\
3) sequence $\varphi^j_n$ converges uniformly to $f$ on $J_j$ as $n\to\infty$;\\
4) if $\omega_f$ is a modulus of continuity of $f$, then 
$|\varphi^j_n(x) - \varphi^j_n(y)|\le \omega_f(2|x-y|)$.

In this case we can define $f_n$ by the formula,
$$
f_n(x)= \left\{
\begin{array}{ll}
f(x) & x\in \bigcup_{i\in \cI}I_i,  \\
\varphi^j_n(x)   & x \in J_j,\ j=1,\ldots,S.
\end{array}\right.
$$
We notice that for sufficiently large $n$ functions $f_n$ satisfy the admissibility conditions D1 and D2, i.e. C1 and C2 from \cite{RySa} as well as OPC and DCC. As a result, due to  \cite[Theorem 3.8]{RySa} for sufficiently large $n$ there exists $u_n$, a unique solution to (\ref{lg}) with datum $f_n$.

The Comparison Principle, Proposition \ref{comparison}, implies that $u_n$ are commonly bounded. Moreover, we claim that the sequence $u_n$ is equicontinuous. Indeed, solutions to (\ref{lg}) with continuous data are constructed in  \cite[Theorem 3.8]{RySa} as uniform limits of solutions to auxiliary problems. Those solutions have a common modulus of continuity expressed in terms of the modulus of continuity of boundary data, see \cite[Lemma 3.6]{RySa}. Since the uniform convergence preserves the modulus of continuity, we deduce that
sequence $\{u_n\}_{n=1}^\infty$ is equicontinuous, because of a common modulus of continuity.
Thus, due to Arzela-Ascoli Theorem from any subsequence $\{u_{n_k}\}_{k=1}^\infty$ we can extract another subsequence  $\{u_{n_{k_l}}\}_{l=1}^\infty$ converging uniformly to $u$. Since functions $u_n$ are continuous, so is $u$. In addition, the uniform convergence implies that 
$$
Tu = \lim_{n\to\infty} Tu_n = \lim_{n\to\infty} f_n =f.
$$
Moreover, by Miranda Theorem, see \cite{miranda}, $u$ is a least gradient function. As a result $u$ is a solution to (\ref{lg}). Its uniqueness follows from \cite{gorny2}.  As a result, we conclude that not only a subsequence  $\{u_{n_{k_l}}\}_{l=1}^\infty$ converges uniformly to $u$, but also $\{u_n\}_{n=1}^\infty$ does.

Due to step 1, we have $|Du_n|(\Omega) \le \diam\Omega\, TV(f_n).$
The definition of the total variation of $f_n$ for every $\epsilon>0$ yields existence of the partition of $\Omega$ such that
$$
TV(f_n) \le \sum_{k=1}^m |f_n(x_{k-1}) - f_n(x_k)| + \epsilon.
$$
We notice that we  only increase the sum, if instead of $x_k,$ $x_{k-1}$ we take the closest points $\xi_k$, $\xi_{k-1}$ 
from $\bX$, which are common for all $f_n.$ Thus,
$$
TV(f_n) \le \sum_{k=1}^m |f_n(\xi_{k-1}) - f_n(\xi_k)| + \epsilon.
$$
Since the points from $\bX$ do not depend on $n$, we may pass to the limit on the right-hand-side,
$$
\limsup_{n\to\infty} TV(f_n)-\epsilon \le \lim_{n\to\infty} \sum_{k=1}^m |f_n(\xi_{k-1}) - f_n(\xi_k)|  = \sum_{k=1}^m |f(\xi_{k-1}) - f(\xi_k)|  \le TV(f) .
$$
On the other hand, the lower semicontinuity of the total variation of $|Du|$ implies
$$
|Du|(\Omega)\le \liminf_{n\to\infty}|Du_n|(\Omega) \le \diam\Omega \limsup_{n\to\infty} TV(f_n)\le \diam\Omega \, TV(f).
$$
Our  the claim follows.
\end{proof}

\section{Construction of solutions to (\ref{lg}) in the case of discontinuous data}\label{sDdata}

Now, we are ready to deal with discontinuous data. Namely, we
study \eqref{lg} when $f \in BV$. In this case $f$ might have at most  countably many jump discontinuity points. 
We consider only $f$ satisfying the admissibility conditions presented in Section \ref{subsec:adm condition}.  The idea of the construction of solutions is to approximate first $f$ from above and below by sequences of continuous data satisfying the admissibility conditions, then using a limiting argument we show that we have a sequence of least gradient functions converging to a solution to \eqref{lg} with trace $f$. 


We will deal separately with the cases when $f$ has finitely many humps and $\d\Omega$ has finitely many sides or one of these two quantities is infinite.

\subsection{Approximation of the data}

We want to approximate $f$ by continuous functions from below and above. A prototype of such approximation is given in \cite{NR}, but here we require special properties of the approximating sequences. 
The technique we use requires that $\d\Omega$ has a finite number of sides and $f$ has a finite number of humps. This is our standing assumption in this subsection.

In particular, we want that the approximation from below preserves minima, while the approximation from above preserves maxima. This statement requires an explanation what is a minimum or maximum for a discontinuous $f$.  
\begin{defn}
Let us suppose that $f\in BV(\d\Omega)$ satisfies (\ref{rd0}). 
We shall say that $f$ has at $x_0\in \d\Omega$ a generalized local maximum (resp. generalized local minimum), if there is a neighborhood $U$ of point $x_0$ such that 
$$
 \sup_{x\in U}f(x) =\limsup_{x\to x_0} f(x) \qquad \left(\hbox{resp. } \inf_{x\in U}f(x) =\liminf_{x\to x_0} f(x)
\right).
$$
\end{defn}
We notice that in  case $f$ is continuous the generalized local minimum is a local minimum (resp. the generalized local maximum is a local maximum). 

Here is our first observation. 
\begin{cor} \label{cr1}
If $f\in BV$ and it satisfies the admissibility conditions D1, D2  and D3, then the generalized local maxima/minima are actually attained at the vertices of $\d\Omega$ or on humps. In particular, $f$ is continuous at the local maxima/minima occurring at vertices.
\end{cor}
\begin{proof}
Let us suppose $x_0$ is a generalized local maximum and it belongs to the relative interior of a side $\ell$. If $f$ suffers a jump discontinuity at $x_0$, then due to D3 this point must belong to the closure of a hump, $\bar I$ where the local maximum is attained at $x_1$. 

Let us suppose that $x_0$ is a vertex, i.e. $x_0\in \ell_1\cap \ell_2$, where $\ell_1$, $\ell_2$ are two sides of $\d\Omega$. The datum $f$ satisfies either D1 or D2 on these sides. However, part (b) of the definition of D2 means that from the point of view of the present argument, we may assume that $f$ satisfies condition D1 on $\ell_1$ and $\ell_2$. In this case, due to D1 $f$ is continuous at $x_0$, hence $f$ attains its generalized local maximum at $x_0$. The argument for generalized local minima is the same.
\end{proof}

We can parameterize $\d\Omega$ using the arc length parameter, $[0,L)\ni s\mapsto x(s)\in \d\Omega$, $L= \cH^1(\d\Omega)$. We may assume that $f$ attains a global minimum at $x(0)$. We can identify $f$ with a function over $[0,L)$. Moreover, we will extend $f$ to the whole line by setting $\tilde f(x) = f(x)$ if $x\in [0,L)$ and $\tilde f(x) = f(0)$ for $x\not\in [0,L)$. We notice that due to Corollary \ref{cr1} function $\tilde f$ is continuous at $x=0$ and $x=L$. 

Subsequently, we will suppress the tilde and we will identify freely $f$ with its extension to $\bR$.   
Now, we can construct sequences approximating our datum.
\begin{lemma}\label{Aprox1}
Let us suppose that $\Omega$ is convex, and its boundary has a finite number of sides. We assume that $f\in BV(\d\Omega)$ satisfies the admissibility conditions D1, D2 and D3. Moreover, $f$ has a finite number of humps. Then, there exist sequences $\{g_n\}_{n=0}^\infty$, $\{h_n\}_{n=0}^\infty$ in  $C(\d\Omega)$ with the following properties:\\ 
(1) For each $x\in\d\Omega$ sequence $\{g_n(x)\}_{n=0}$ (resp. $\{h_n(x)\}_{n=0}$) is increasing (resp. decreasing)  and $g_n(x)\le f(x) \le h_n(x)$.\\
(2) For all $x\in\d\Omega$, which are  points of continuity of $f$, we have  
$\displaystyle{\lim_{n\to \infty} g_n(x) = f(x) =  \lim_{n\to \infty} h_n(x)}$.\\
(3) If $f$ has a local minimum (resp. maximum) at $x_0\in \d\Omega$, then
$g_n(x_0) = f(x_0)$ (resp. $h_n(x_0)= f(x_0)$.\\
(4) If $\ell$ is a side of $\d\Omega$ and $f|_\ell$ satisfies D1, so do $g_n|_\ell$ and $h_n|_\ell$.
\end{lemma}

\begin{proof}
Since $f\in BV(\bR)$, then  its distributional derivative $f'$ is a measure. Due to the  Hahn decomposition theorem, we have $f' = (f')^+ - (f')^-$, where  $(f')^+$ and $(f')^-$ are mutually singular positive measures. Moreover, $$
\supp (f')^+ \cap \supp (f')^- =: Z
$$
consists of a finite number of points. More precisely, due to Corollary \ref{cr1} elements of $Z$ are vertices of $\d\Omega$. We set,
$$
f^+ (x) = (f')^+([0,x]) + f(0),\qquad
f^-(x) = (f')^- ([0,x]).
$$
Then, we  can write 
$f=f^+ - f^-$, where $f^+, f^-$ are increasing functions. Subsequently, we will regularize $f^+$ and $f^-$.

Let $\varphi_{1/n}$ be the standard approximation to the identity with support in $[-\frac 1n, \frac 1n]$. We consider $n\in \bN$, such that
$$
\min\{\alpha,\beta, \gamma\}  \ge \frac 8 n,
$$
where
\begin{align*}
&\alpha = \min\{|\ell_i|, i=1,\ldots, K\},\\
&\beta  = \min\{|I_j|, j\in \cI\},\\
&\gamma = \min\{|J|:\ J \hbox{ is a connected component of }\ell\setminus \bigcup_{j\in \cI} I_j, \ell\ \hbox{is a side of }\d\Omega\} .
\end{align*}
We consider the mollified sequences
$$
\bar g_n(x)= f^+\ast \varphi_{1/n}\left(x-\frac{1}{n}\right)- 
f^-\ast \varphi_{1/n}\left(x +\frac{1}{n}\right),
$$
$$
\bar h_n(x)= f^+\ast \varphi_{1/n}\left(x+\frac{1}{n}\right)- 
f^-\ast \varphi_{1/n}\left(x -\frac{1}{n}\right).
$$
Functions $\bar g_n$, $\bar h_n$ are continuous. Moreover,  sequences $\bar g_n(x)$ and $\bar h_n(x)$ converge  to $f(x)$ at continuity points, 
because of general properties of the convolution.  

Below, we present an analysis of $\bar g_n$ and a construction of $g_n$. The construction of $h_n$ follows similarly.

We notice that if $[a,b]$ is an interval such that $f|_{[a,b]}$ is monotone increasing, then $\bar g_n'\ge 0$ on $[a, b-\frac 2n]$. Indeed,
for $x\in [a,b-\frac 2n]$, then we have
$$
\bar g_n'(x) = (f')^+ * \varphi_{1/n}\left(x -\frac 1n\right)
- (f')^- * \varphi_{1/n}\left(x +\frac 1n\right).
$$
The first term in non-negative by definition, while the second one vanishes, because  $\supp \varphi_{1/n}(\cdot + x+\frac1n) $ does not intersect the support of $(f')^-$ for $x\in (a,b-\frac 2n]$.

The same argument shows that if $f$ is monotone decreasing on $[b,d]$, then $\bar g_n'\le 0$ on $[b+\frac 2n,d]$.

Let us suppose that $x_0$ is a local minimum of $f$, which does not belong to any hump. Then, due to Corollary \ref{cr1}, point $x_0$ must be an endpoint of intervals $\ell_1$, $\ell_2$, where $f$ is continuous and there are intervals $J_i\subset\ell_i$, $i=1,2,$ such that $f$ satisfies D1 on $J_i$, $i=1,2$. Due to the continuity of $f$ at $x_0$, we notice that $\bar g_n(x_0) = f(x_0).$ Indeed, we notice that $f^+$ restricted to $(x_0-\frac2n, x_0]$ is equal to $f^+(x_0)$, while $f^-$  restricted to $[x_0, x_0+\frac2n)$ is equal to $f^-(x_0)$. Due to the continuity of $f$, the values of $f^\pm(x_0)$ are well-defined. Hence, 
the definition of $\bar g_n$ and the properties of convolution  imply that $$
\bar g_n(x_0)  = f^+(x_0) - f^-(x_0) = f(x_0).
$$

We are going to adjust $\bar g_n$, so that the modification, $g_n$, satisfies claim (4). Let us suppose that $f$ is increasing on $[a,b]$ and decreasing on $[b,d]$, i.e., $f$ has a strict local maximum at $b$. We are going to modify $\bar g_n$ on $[b-\frac3n, b+\frac3n]$, so that the modified function $g_n$ is increasing on $[a,b]$ and decreasing on $[b,d]$. 

Due to the definition of $f^\pm$ we have 
$$
f(x)\ge \min\left\{ f^+\left(b-\frac2n\right), f^-\left(b+\frac2n\right)\right\}\qquad\hbox{for all } x\in \left[b-\frac2n,b+\frac2n\right].
$$
As a result 
$$
\bar g_n(x)\ge \min\left\{ \bar g_n\left(b-\frac2n\right),\bar g_n \left(b+\frac2n\right)\right\}=:\gamma \qquad\hbox{for all } x\in \left[b-\frac2n,b+\frac2n\right].
$$
Keeping this in mind, we set
$$
g_n(x) = \left\{
\begin{array}{ll}
\bar g_n(x)     & x\in [a, b-\frac3n)\cup [b+\frac3n, d], \\
\frac n3 (\gamma - \bar g_n(b-\frac3n))(x-b ) + \gamma     & x\in [b -\frac3n, b],\\
\gamma & x = b,\\
\frac n3 (\bar g_n(b+\frac3n) - \gamma)(x-b) + \gamma&
x\in [b+\frac3n, d].
\end{array}
\right.
$$
Due to the definition of $\gamma$ function $g_n$ is increasing  on $[a,b]$ and decreasing on $[b,d]$, hence condition (4) is satisfied. Let us notice that this argument is correct when $b$ is a vertex  of $\d\Omega$ or a hump endpoint. In the last case $f^-$ is constant in the neighborhood of $b$.

We have to check that properties (1)--(3) hold for $g_n$. Let us fix $n$ and take any point $x\in [a,b]$. We have to show that
$$
g_n(x) \le g_{n+1}(x).
$$
First, let us consider the case $x\in[a, b-\frac3n]$. The definition of $\bar g_n(x)$ yields, $g_n(x) = \bar g_n(x)$, then after  the change of variable $z=n(x - \frac 1n - y)$, we get 
\begin{eqnarray*}
 g_n(x) + f^-(a+\frac1n)&=&\int_{\mathbb R} f^+\left(x-\frac{1}{n}(1+z)\right)\varphi_1(z)\,dz 
 =\int_{-1}^{1}f^+\left(x-\frac{1}{n}(1+z)\right)\varphi_1(z)\,dz \\
 &\leq& \int_{-1}^1 f^+\left(x-\frac{1}{n+1}(1+z)\right)\varphi_1(z)\,dz=g_{n+1}(x)
 + f^-(a+\frac1n).
\end{eqnarray*}
If $x\in [b+\frac 3n, d]$, then we see,
\begin{eqnarray*}
 g_n(x)- f^+(b+\frac2n)&=& -  \int_{-1}^{1}f^-\left(x+\frac{1}{n}(1-z)\right)\varphi_1(z)\,dz \\
 &\leq&-\int_{-1}^1 f^-\left(x+\frac{1}{n+1}(1-z)\right)\varphi_1(z)\,dz=g_{n+1}(x) -
 f^+(b+\frac2n).
\end{eqnarray*}
The case $x\in [b - \frac3n, b+ \frac 3n]$ is obvious due to the definition of $g_n$ and $\gamma$. Thus, (1) holds for sequence $\{g_n\}_{n=1}^\infty$.

We take care of part (2) for $g_n$. Let us take interval $[a,b]$ (resp. $[b,d]$) such that $f$ restricted to it is increasing (resp. decreasing). If $x\in [a,b)$ (resp. $x\in (b,d]$), then for sufficiently large $n$ we see that $x\in [a, b-\frac3n]$ (resp. $x\in [b+\frac 3n, d]$). In this case $g_n(x) = \bar g_n(x)$ and we shall see that $\bar g_n(x) \to f(x)$, as $n\to\infty$ if $x$ is a point of continuity of $f$. Indeed, for a given $\epsilon>0$, we can find $\delta>0$ such that if $|z|<\delta$, then $|f(x+z)-f(x)|<\epsilon$. Moreover, for our choice of $x$, we have $\bar g_n(x) = f^+* \varphi_{1/n}(x-\frac1n)$ or $\bar g_n(x) = - f^-* \varphi_{1/n}(x+\frac1n)$. We will consider only the first possibility, then we have
$$
|\bar g_n(x) - f(x)| \le \int_{-1/n}^{1/n} | f^+(x-\frac1n - y) - f^+(x))\varphi_{1/n}(y)| < \epsilon
$$
for $n>2/\delta$.

Let us suppose that $b$ is a point of continuity of $f$ and $f$ is increasing on $[a,b]$. Checking that $g_n(b)$ converges to $f(b)$ is done as in the previous case, but is based on the fact that $\bar g_n(b \pm \frac 2n)$ goes to $f(b)$, when $n\to \infty.$ The details are omitted.

Property (3) has been already proved for $\bar g_n$, provided that $f$ attains a local minimum at $x_0$, hence $f(x_0) = g_n(x_0)$ and our claim follows.

We conclude that, we have established our claim for sequence $\{g_n\}_{n=1}^\infty$. The obvious modification of this argument is sufficient to show the corresponding results for $\{h_n\}_{n=1}^\infty$.
\end{proof}



\subsection{Data with a finite number of humps and $\d\Omega$ with a finite number of sides} \label{S3.1}

We treat here the case of polygons with a finite number of sides. Moreover, we restrict our attention to $f\in BV(\d\Omega)$ with a finite number of humps. These will be our standing assumption in this subsection. 
In the previous subsection we constructed approximations of $f$ from below and above. We will show here that these sequences  satisfy the assumption of the existence theorem in \cite[Theorem 3.8]{RySa}.


\begin{lemma}\label{gnadm2}
 Let us suppose that $f\in BV(\d\Omega)$ satisfies the admissibility condition D2 and DCC. Then, the approximating sequences of functions, $g_n$ and $h_n$, constructed in Lemma \ref{Aprox1}, satisfy  admissibility condition D2.
\end{lemma}
\begin{proof}
We take a 
hump of $f$, i.e. $\bar I_i =[a_i, b_i]$ and $f(a_i, b_i) = e_i$. 
We have to investigate the behavior of the humps of $g_n$ and $h_n$ for large $n$. Since the argument for both type of functions is the same, we will present the argument for $g_n$ only. For the sake of definiteness, we may assume that $f$ attains a maximum on $I_i$, the case of $f$ having a  minimum there is similar.
By Lemma \ref{Aprox1} $g_n(a_i)\le f(a_i^-)$, (recall that vectors $b_i -a_i $ agree with the positive orientation of $\d\Omega$).
Hence,
$[\alpha^n_i, \beta^n_i] = g_n^{-1}(e_i) \cap I_i \subsetneq P(e_i) \cap I_i$.
Actually, by the construction of $\bar g_n$ and $g_n$ we deduce that $\alpha_i^n= a_i + \frac 2n$ and $\beta_i^n= b_i - \frac 2n$, for sufficiently large $n$. The arithmetic operations are performed  on the arclength parameter, so that they are well-defined.

Hence,
$
\alpha^n_i \to a_i$ and 
$\beta^n_i \to b_i.
$
We have to investigate the behavior of $y^n_i$ corresponding to $\alpha^n_i$ (resp. $z^n_i$ corresponding to $\beta^n_i)$. Their behavior is similar, that is why we can look only at $y^n_i$. 

Our first observation in this direction is that $y_i$ cannot be a strict local maximum  of $f$, 
because this contradicts DCC. 
Thus, there exists $U$, a neighborhood of $y_i$, such that $f|_U$ is monotone. We have the following possibilities: \\
(i) there exists an interval $J=(p_i,q_i)\subset \ell$ containing $y_i$, where $\ell$ is a side of $\Omega$ and such that $f$ restricted to $(p_i,q_i)$ is monotone, moreover, the tangent vector $p_i-q_i$ agrees with the positive orientation of $\d\Omega$; \\ 
(ii) $y_i$ is a vertex of $\d\Omega$.


We shall discuss the positively oriented arc contained in $\d\Omega$ with endpoints $q_i, a_i$ and containing $p_i, y_i$.
We face the following possibilities of the behavior of $f$ at $y_i$:\\
(a) $f$ is  discontinuous at $y_i$ and $f|_{{[q_i,y_i)}} =  e_i$ and $f(p_i)<  e_i$;\\
(b) 
$f$ is  discontinuous at $y_i$ and 
for all $x\in{(q_i,y_i)}$, we have $f(x) > f(y^-_i)$ ;\\
(c) $f$ is continuous at $y_i$.

We remark that (a)  corresponds to the case when $y_i$ may belong to a closure of a hump $I_j$. 

First, we consider  (a).  Since $f(y^+)<e_i$, then due to the properties of the definition of $g_n$, we conclude that $g_n(x)<e_i$ for $x\in (y_i-\frac 2n, y_i]$ and $g_n(y_i- \frac 2n) = e_i$. Hence, $y^n_i = y_i- \frac 2n$.

If (b) occurs, then we proceed as follows. We notice that by  definitions of $P(e_i)$, see Definition \ref{image set} and that of $g_n$ we have,
\begin{equation}\label{r3.4}
g_n(y_i) \le f(y_i^+) \le e_i {\le f(y^-_i)},\qquad \hbox{{and}}\qquad e_i = g_n(y_i^n).    
\end{equation}
Due to assumption (c) and the convolution properties we deduces that 
$g_n(y_i) <g_n(y^n_i)$ implying that $y_i > y^n_i$. We claim that $y^n_i$ converges to $y_i$, when $n\to\infty$. Indeed, if there were a subsequence  $y^{n_l}_i {\le y_i -\delta}$, 
then, after assuming that $y-\delta$ is a continuity point of $f$, we see that
$$
g_{n_l}(y_i-\delta) \le e <f(y_i-\delta)\le f (y^{n_l}_i) .
$$
The LHS goes to $f(y_i -\delta)$. Thus, we reached a contradiction. In other words, $y^n_i \to y_i$.

We  claim that (c) follows from (a) and (b), where we described possible types of behavior of $f$ in a neighborhood of $y$. Now, we consider $f$ continuous at $y_i$. If $y_i$ is in a relative interior of an interval, where $f$ is constant, e.g. on a hump, then there is nothing to do. Otherwise we can reuse (\ref{r3.4}). If we realise we have not based our argument on $f(y^+_i) \neq f(y^-_i)$, i.e., $f(\zeta)< f(y_i)$ for $\zeta<y_i$ close to $y_i$, but rather on non-constancy of $f$, then we see that our claim follows. 



Thus, we showed that 
$$
\alpha^n_i \to a_i, \quad \beta^n_i \to b_i, \quad y^n_i \to y_i, \quad z^n_i \to z_i \qquad\hbox{when }n\to \infty.
$$
In this way, we deduce that
(\ref{DA}) holds for sufficiently large $n$.
%
\end{proof}

%

\begin{lemma}\label{moDCC}
Let us suppose that $f\in BV(\Omega)$ satisfies D1, D2 and D3, see Definitions \ref{admDC1}--\ref{admDC3},  as well as  DCC (\ref{dcc1}-\ref{dcc2}). We also assume that $\Omega$ has a finite number of sides and $f$ has a finite number of humps. Then, the approximating sequences $\{g_n\}_{n=1}^\infty$,  $\{h_n\}_{n=1}^\infty$ satisfy  DCC too.
\end{lemma}
\begin{proof}
It is sufficient to consider the case of $f$ attaining a local maximum on a hump $I_i$.
The other case is handled in a similar manner.

{\it Step 1.} We will first state  (\ref{dcc1}) for the approximating functions $g_n$ and $h_n$. For this purpose we introduce more notation. If $\bar I = [a,b]$ is the closure of a hump of $f$, then $[a^g_n, b^g_n]$ (resp. $[a^h_n, b^h_n]$) is the closure of a corresponding hump of $g_n$ (resp. $h_n)$. We shall write
$$
\gamma^n_g = \overline{y^n_gz^n_g}_{a^n_g b^n_g}, \qquad
\gamma^n_h = \overline{y^n_hz^n_h}_{a^n_h b^n_h}.
$$
Keeping this in mind, we have to show that
\begin{equation}\label{nier}
 \inf_{x\in \gamma^n_g} g_n(x)\ge g_n((a^n_g, b^n_g)),\qquad
\inf_{x\in \gamma^n_h} h_n(x)\ge h_n((a^n_h, b^n_h)).
\end{equation}
where $(a^n_g, b^n_g)$ (resp. $(a^n_h, b^n_h)$) is the interior of a hump of $g_n$ (resp. $h_n$) approximating $(a,b)$. 
We recall that
$$
g_n((a^n_g, b^n_g)) = h_n((a^n_h, b^n_h)) = f((a,b)) = e.
$$

{\it Step 2.}  
We notice that there exists $\delta>0$ such that $f$ restricted to $B(y,\delta)\cap\d\Omega$ (resp. $B(z,\delta)\cap\d\Omega$) is monotone. Indeed, this follows from (\ref{dcc1}) implying that $f$ cannot have an extremum at $y$ (resp. $z$) and the fact that $f$, satisfying the assumptions of the Lemma, may change its type of monotonicity only a finite number of times. 

In particular, this observation implies that 
$$
\gamma^n_g \subset \overline{yz}_{ab}  \subset \gamma^n_h.
$$

{\it Step 3.}  We will show the second of the inequalities in (\ref{nier}). 
Since $h_n\ge f$, we notice that due to (\ref{dcc1}) we have
$$
\inf_{\overline{yz}_{ab}} h_n\ge \inf_{\overline{yz}_{ab}} f\ge e. $$
Moreover, due to Step 2, we see
$$
\inf_{x\in \gamma^n_h\setminus \overline{yz}_{ab}} h_n(x)\ge h_n((a^n_h, b^n_h)) =e.
$$
As a result,  our claim holds.

{\it Step 4.}  Now, we will show the first  inequality in (\ref{nier}). By definition $f\ge g_n$, in particular for $x\in \gamma^n_g{\subset \overline{yz}_{ab}} $, we have
$$
f(x) \ge e = 
g_n(y^n_g)= g_n(z^n_g).
$$
Let us consider $x_0\in \gamma^n_g$, any local 
minimum of $f$. 
Due to Lemma \ref{Aprox1} (3) we have that $g_n(x_0) = f(x_0)$.
Hence, our claim follows.
\end{proof}
\begin{lemma}\label{moOPC}
Let us suppose that $f\in BV(\Omega)$ satisfies D1, D2 and D3, see Definitions \ref{admDC1} and \ref{admDC2}, as well as the OPC, (\ref{OPC}). We also assume that $\Omega$ has a finite number of sides and $f$ has a finite number of humps. Then, the approximating sequences $\{g_n\}_{n=1}^\infty$,  $\{h_n\}_{n=1}^\infty$ satisfy OPC too.
\end{lemma}
\begin{proof}
We will discuss the case of the sequence $\{g_n\}_{n=1}^\infty$ only,  because the argument for  $\{h_n\}_{n=1}^\infty$ is the same. We  consider humps $[a_1,b_1]$, $[a_2,b_2]$ and the corresponding points $y_i$, $z_i$, $i=1,2$. For any $g_n$ we find humps $[a^{g_n}_i, b^{g_n}_i]\subset [a_i,b_i]$ and  the corresponding points $y^{g_n}_i$, $z^{g_n}_i$, $i=1,2$. 
We have seen in the course of proof bf Lemma \ref{Aprox1} that 
$$
\lim_{n\to\infty} a^{g_n}_i = \lim_{n\to\infty} a^{h_n}_i =a,\qquad
\lim_{n\to\infty} b^{g_n}_i = \lim_{n\to\infty} b^{h_n}_i =b,\quad i=1,2,
$$
$$
\lim_{n\to\infty} y^{g_n}_i = y_i,\qquad \lim_{n\to\infty} z^{g_n}_i,\quad i=1,2.
$$
Let us now suppose that our claim is not true. i.e. there are infinitely many intervals $J_1^n$, $J_2^n$ of the form,
$$
J_1^n = [a^{g_n}_1, y^{g_n}_1],\quad \hbox{or}
\quad J_1^n = [b^{g_n}_1, z^{g_n}_1]
$$
$$
J_2^n = [a^{g_n}_2, y^{g_n}_2],\quad \hbox{or}
\quad J_2^n = [b^{g_n}_2, z^{g_n}_2],
$$
with not empty intersection, $\xi^n \in J_1^n\cap J_2^n$. The exact form of $J^n_i$, $i=1,2$ does not matter. We may write $J_1^n=[\alpha^n,\beta^n]$, $J_2^n = [\gamma^n,\delta^n]$, hence points $\alpha^n,\beta^n, \xi^n$ are co-linear so are the points $\gamma^n,\delta^n, \xi^n$. 
We may write this as
\begin{equation}\label{req}
Q(\alpha^n,\beta^n, \xi^n) = 0,\qquad Q(\gamma^n,\delta^n, \xi^n) =0,
\end{equation}
where $Q$ is a polynomial. Since sequence $\xi_n$ is bounded we may extract a convergent subsequence (without  relabeling), $\xi^n \to \xi$. Hence, the continuity of $Q$ implies that,

$$
Q(\alpha,\beta, \xi) = 0,\qquad Q(\gamma,\delta, \xi) =0.
$$
But this means that
$$
([a_1,y_1] \cup [b_1,z_1]) \cap ([a_2,y_2] \cup [b_2,z_2])\neq \emptyset.
$$
In other words, we reached a contradiction with (\ref{dcc1}). Our claim follows.
\end{proof}



We are ready to prove the main result of this section.
 \begin{theorem}\label{Main2}
Let us suppose that 
$\Omega$ is convex and $\d\Omega$ is a polygon. Moreover,
$f\in BV(\d\Omega)$
$f$ satisfies the admissibility conditions D1, D2,  D3,  as well as the complementary conditions OPC, (\ref{OPC}), and DCC, (\ref{dcc1}--\ref{dcc2}).
If $f$ has finitely many humps, then there exists a solution $u$ to the least gradient problem (\ref{lg}). In addition, 
\begin{equation}\label{rtw37}
\| u \|_{L^\infty}\le \| f\|_{L^\infty}\qquad\hbox{and}\qquad\int_\Omega |Du| \le \diam \Omega \, TV(f).
\end{equation}

\end{theorem}
\begin{proof}
Define $C_f=\{x\in \d\Omega: \text{$f$ is continuous at $x$}\}$. By Lemmas \ref{Aprox1}, \ref{gnadm2}, \ref{moDCC} and \ref{moOPC} we construct a decreasing sequence of continuous functions $h_n$ such that 
$h_n\to f$ in $C_f$. For each $h_n$, by \cite[Theorem 3.8]{RySa}, there exists a continuous solution $u_n$ to \eqref{lg} on $\Omega$, with trace $h_n$. By the comparison 
principle in Proposition \ref{comparison}, we deduce that $u_n$ is a decreasing sequence. Then, these functions converge to a function $u$ at every point $x\in \Omega$. By 
\cite{miranda}, $u$ is a least gradient function. 

By a similar token, by Lemmas \ref{Aprox1}, \ref{gnadm2}, \ref{moDCC} and \ref{moOPC} we construct an increasing sequence of continuous functions $g_n$ such that 
$g_n\to f$ in $C_f$. For each $g_n$, by \cite[Theorem 3.8]{RySa}, there exists a continuous solution $v_n$ to \eqref{lg} on $\Omega$ with trace $g_n$. By the comparison 
principle, in Proposition \ref{comparison}, we have that $v_n$ is an increasing sequence, converging to a function $v$ at every point $x\in \Omega$. By 
\cite{miranda}, $v$ is a least gradient function.

We shall prove that 
$$
Tu(x) = f(x) = T v(x)\qquad\hbox{for } x\in C_f.
$$
We claim $f(x)\ge Tu(x)$ for $x\in C_f$.
If it were otherwise, $Tu(x)=t>f(x)=\tau$, then there would exist $s  \in (\tau,t)$.
Since $h_n(x)\to f(x)$ then, there exists $N$ such that $h_N(x)<s$. By the continuity of $u_N$, we gather that $u_N(y)<s$ for every $y$ in a neighborhood of $x$ in $\Omega$, we may assume that this is a ball $B(x,r).$
Since $u_n$ is a decreasing sequence, then $u_n(y)<s$ for all $n\geq N$. Letting $n\to \infty$, we get that $u(y)\leq s$ for all $y\in B(x,r)\cap \Omega$, contradicting the fact that $Tu(x)=t>s$. 

Now, we consider  sequence $g_n$  converging to $f$ from below and the corresponding solutions $v_n$ to the least gradient problem. The sequence $v_n$ converges to a least gradient function $v$.

The same argument as above implies that
$ Tv(x) \ge f(x)$.
Since we automatically have that 
$$
u_n(x) \ge u(x) \ge v(x) \ge v_n$$
we deduce from the above inequalities that
$$
f(x) = Tu(x) = Tv(x) = f(x)\qquad\hbox{for }x\in C_f.
$$
Since $C_f\subset \Omega$ has the full measure, we deduce that $T u = f$, as desired. The same argument yields $T v = f$.

We can draw further conclusions from the comparison principle. Namely, we infer directly that
$$
\max h_n \ge u_n \ge v_n \ge \min g_n.
$$
By the definition of $h_n$ and $g_n$ we deduce that
$$
\| f\|_{L^\infty} \ge \max h_n \ge \min g_n \ge - \| f\|_{L^\infty}.
$$
Hence,
$$
\| f\|_{L^\infty} \ge \|u\|_{L^\infty}\quad\hbox{ and } \quad
\| f\|_{L^\infty}\ge \| v\|_{L^\infty}.
$$

It remains to prove the estimate on the total variation of solutions. Since $g_n$ is an increasing sequence of continuous functions, satisfying the admissibility conditions, DCC and OPC, then due to Lemma \ref{DuleTV} we have
$$
\int_\Omega |Dv_n| \le \diam\Omega\, TV(g_n).
$$
The lower semicontinuity of the total variation yields $|D u|(\Omega)\le \liminf_{n\to\infty} |Dv_n|(\Omega)$. We have to estimate $TV(g_n).$ 
If we recall  
the definition of $g_n$, then we see that,  
$$
TV(g_n) = BV(g_n) =\int_{[0,L)\setminus E_n} | D \bar g_n | + \int_{E_n} |D g_n|,
$$
where 
$$
E_n = \bigcup \left[v_j - \frac3n , v_j+\frac3n\right],
$$
where $v_j$'s are those vertices, where $f$ attains a local maximum.
We recall that we identify a vertex $v_j$ with a point in $[0,L)$ denoted in the same way. Then, we have
$$
\int_{[0,L)\setminus E_n} | D \bar g_n | \le 
 \int_0^L\left| (Df)^+*\varphi_{1/n})(\cdot - \frac1n)\right| + \int_0^L\left| (Df)^-*\varphi_{1/n})(\cdot + \frac1n)\right| .
 $$
 This is so because $Df^+$ and $Df^-$ have disjoint supports. Now,  \cite[Theorem 2.2, page 42]{AFP} implies that the right-hand-side above converges to $TV(f)$.

Now, we  shall show that  $\int_{E_n} |D g_n|$ goes to zero, when $n\to \infty$. Indeed,
\begin{eqnarray*}
\int_{v_j-\frac 3n}^{v_j+\frac 3n} |D g_n| &=& \left(\gamma_n - \bar g_n\left(v_j - \frac 3n\right)\right) + 
\left(\gamma_n - \bar g_n\left(v_j + \frac 2n\right)\right) \\ &\le &
\left(\bar g_n\left(v_j - \frac 2n\right) - \bar g_n\left(v_j - \frac 3n\right) \right)+
\left(\bar g_n\left(v_j + \frac 2n\right) - \bar g_n\left(v_j + \frac 3n\right)\right).
\end{eqnarray*}
Since $f$ is continuous at each point $v_j$, we conclude that the right-hand-side above converges to zero.  Hence,
$$
\int_\Omega |Dv| \le \diam\Omega\, TV(f).
$$
The same argument applies to $h_n$'s and $u$ yielding,
$$
\int_\Omega |Du| \le \diam\Omega\, TV(f).
$$
Even if there are other solution, they satisfy the above inequality, because if $w$  violates it, then it cannot be a minimizer.
Our claim follows.
\end{proof}

\begin{remark}
 In the course of the above proof, we constructed two solutions $u$ and $v$ to the least gradient problem having the trace at the boundary. However, they need not be equal. The Brothers example, see \cite{mazon}, shows that actually  the non-uniqueness may be quite severe.
\end{remark}

We state a useful observation, which is known for continuous data.
\begin{cor}\label{c39}
Let us suppose that the hypothesis of the previous theorem hold and $[a,b]$ is a closure of a hump $I.$ If $y$ and $z$ are defined by (\ref{ddfyz}), then $u(x)=e$ for all $x$ in the interior of $\conv(a,b,y,z)$,  where $e= f(I)$ and $\conv(a,b,y,z)$ stands for the convex envelope of point $a,b,y$ and $z$.
\end{cor}
\begin{proof}
We take $g_n$ approximating $f$ and a hump $[\alpha_n,\beta_n]$ corresponding to $I$, i.e. 
$$
\lim_{n\to\infty}\alpha_n=a,\qquad \lim_{n\to\infty} \beta_n = b.
$$
We also take $y_n$ (resp. $z_n$) corresponding to $y$ (resp. $z$). We notice that the proof of \cite[Lemma 3.5]{RySa} combined with \cite[Proposition 3.9]{RySa} yields that,
$$
u_n ([\alpha_n, y_n]) = e = u_n([\beta_n,  z_n]),
$$
where $u_n$ is the solution to (\ref{lg}) with datum $g_n.$
By construction we have $u_n([\alpha_n, \beta_n]) =e.$

We set $\tilde Q_n:=\hbox{int}(\conv(\alpha_n,\beta_n, y_n, z_n))$. We  claim that for all $x\in \tilde Q_n$ we have $u_n(x) =  e$. If this were not true, then we could find $\tilde u_n$ with the total variation smaller than that o $u_n$. Indeed, we set
$$
\tilde u_n(x)=
\left\{
\begin{array}{ll}
u_n(x)   & x\in \Omega\setminus \tilde Q_n, \\
e  &  x\in \tilde Q_n.
\end{array}
\right.
$$
Then, we compute
$$
\int_\Omega | D \tilde u_n | = \int_{\Omega\setminus\tilde Q_n} | D \tilde u_n | + \int_{[y_n,z_n]}(T u_n -e)\,d\cH^1,
$$
where 
$T  u_n$ denotes the trace of $ u_n$ over $[y_n,z_n]$. 
Due to the DCC we know that $T u_n -e \ge 0.$

At the same time continuity of $u_n$ 
implies that 
$$
\int_\Omega |Du_n| = \int_{\Omega\setminus\tilde Q_n} |Du_n|
+ \int_{\tilde Q_n} |Du_n|.
$$
The co-area formula yields,
$$
\int_{\Omega\setminus\tilde Q_n} |Du_n| =
\int_e^{m} P(\{u_n \ge t\},\tilde Q_n)\,dt,
$$
where $m = \max_{[y_n,z_n]} Tu_n$. Since 
$$
\{ u_n|_{\overline {\tilde Q}_n} \ge t \} \cap ([\alpha_n, y_n]
\cup [\alpha_n, \beta_n] \cup [\beta_n, z_n]) = \emptyset
$$
for $t>e$, we deduce that
\begin{equation}\label{rc39}
P(\{u_n \ge t\},\tilde Q_n) 
> \cH^1( \{u_n \ge t\} \cap [y_n,z_n])
\end{equation}
for $t$ from a set of positive measure contained in $(e,m]$. If we integrate (\ref{rc39}) over $[e,m]$, then we see
$$
\int_{\tilde Q_n}|Du_n| >
\int_e^m \cH^1( \{u_n \ge t\} \cap [y_n,z_n])\, dt
=  \int_{[y_n,z_n]}(T u_n(x) -e)\,d\cH^1(x).
$$
But this contradicts minimality of $u_n$, contrary to our assumption. As a result, $u_n(x) = e$ for all  $x\in \tilde Q_n.$ If $x$ is in the interior of $\conv(a,b,y,z)$, then for a sufficiently big, $n$ we have $x\in Q_n$.
Since $u$ is a pointwise limit of $u_n$, then we deduce,
$$
e = \lim_{n\to \infty} u_n(x) = u(x),
$$
as desired.
\end{proof}

\subsection{Discontinuous data when $\d\Omega$ has infinitely many sides and a finite number of humps}

In this section we prove our second main results. It is an analogue of \cite[Theorem 3.10]{RySa}, where we had continuous data.
\begin{theorem}\label{td-nsk}
Let us suppose that $\Omega$ is an open, bounded and convex set, whose boundary,  $\d\Omega$ has an infinite  number of sides. We assume that the number of humps  is finite. In addition,
there exists exactly one point $p_0$ being an endpoint of a side $\ell_0$, which is an accumulation point of the sides of $\d\Omega$. We assume  that $f\in BV(\d\Omega)$ and it satisfies the
admissibility conditions D1 or  D2 on all sides of $\d\Omega$, D3  and the Order Preserving Condition (\ref{OPC}) and the Data Consistency Condition, (\ref{dcc1}-\ref{dcc2}).
Finally, 
$f$ attains a strict local maximum or minimum at $p_0$, which is a continuity point of $f$ and 
there is $\rho>0$, such that $f$ restricted to each component of $(B(p_0,\rho)\cap \d\Omega) \setminus\{p_0\}$ is monotone. Then, problem (\ref{lg}) has a solution  $u\in BV(\Omega)$ and
\begin{equation}\label{rtw310}
\| u \|_{L^\infty} \le \| f \|_{L^\infty}  \qquad\hbox{and}\qquad
\int_\Omega |Du | \le \diam \Omega \, TV(f).
\end{equation}
\end{theorem}
\begin{remark}
  Here, Corollary \ref{cr1} does not apply, because of the infinite number of sides.
\end{remark}
\begin{proof}
The argument is based on the approximation of the problem in question by problems we dealt with in the proof of Theorem \ref{Main2}. We used a similar technique in  \cite[Theorem 3.10]{RySa}.
Here, however, we face an additional difficulty, which is the  lack of uniqueness of solutions to (\ref{lg}).

{\it Step 1.} We begin with a construction of a sequence of convex sets, $\Omega_n$, such that $\d\Omega_n$ is a polygon having a finite number of sides. For $\rho$ given in the statement of the theorem, we consider all sides of $\d\Omega$, $\{\ell_k\}_{k=1}^\infty$, contained in $B(p_0,\rho)$. Possibly after further restriction of $\rho$, we may assume that if $x_0\in (B(p_0,\rho)\setminus B(p_0,\frac{\rho}{2}))\cap\ell_0$  is a continuity point, then $y_0\in\d\Omega$, the closest point to $x_0$ such that $f(x_0) = f(y_0)$, belongs to $B(p_0,\rho)\setminus\ell_0 $ and $y_0$ is also a continuity point. 

We define sequences $\{x_n\}_{n=1}^\infty$ and $\{y_n\}_{n=1}^\infty$ inductively.
Due to our assumptions on the behavior of $f$ near $p_0$, we can pick
$x_n\in B(p_0,\frac{\rho}{2^n})\setminus B(p_0,\frac{\rho}{2^{n+1}})\cap\ell_0$ a continuity point of $f$ and $y_n\in \bigcup_{k=1}^\infty \ell_k$, $|y_n-p_0|<|y_{n-1}-p_0|$,  a continuity point of $f$, and $f(x_n)=f(y_n)$.  

We set $L_n=[x_n,y_n]$ and $H(L_n,p_0)$ to be the closed half-plane containing $p_0$ whose boundary contains the line segment $L_n$.
We define $\Omega_n=\Omega\setminus H(L_n, p_0)$, $n \in \mathbb{N}$
and
$$
f_n(x)=
\begin{cases} f(x), & x\in \partial \Omega\cap \partial \Omega_{n},\\ f(x_{n}), & x\in L_n .
\end{cases}
$$ 
We see that D2 and  D3 hold and we claim that $f_n$ satisfies the D1  admissibility conditions. Indeed, we could consider $\rho$ so small that there is no hump in the ball $B(p_0,\rho)$. As a result $f|_{B(p_0,\rho)\cap \ell}$ is monotone. The side of $\d\Omega$ containing $y_n$ will be called $\ell_{k_n}$, because not all sides $\ell_k$ need to be selected. In any case,  $f$ restricted to $\ell_{k_n}$ is monotone. Since $f_n = f(x_n)$ on $L_n$ we see that $f_n$ satisfies D1 on $L_n$. 
Moreover, by the choice of $x_n$ and $y_n$, functions $f_n$ satisfy the Order Preserving and Data Consistency  Conditions. 

We have to perform an additional step before invoking  Theorem \ref{Main2}. Let us denote by $L(\ell_i)$ the line containing segment $\ell_i$. We set  $p_n = L(\ell_0) \cap L(\ell_{k_n}).$ Furthermore, if we denote a convex envelope of set $A$ by $\conv (A)$, then we define a convex set  $\widetilde\Omega_n = \Omega_n \cup \conv(p_n, L_n)$ 
and
$$
\tilde f_n(x)=
\begin{cases} f_n(x) & x\in \partial \widetilde\Omega_n\cap \partial \Omega_{n},\\ f(x_{n}) & x\in \p \widetilde\Omega_n\setminus \partial \Omega_{n}.
\end{cases}
$$ 
{\it Step 2.} The above
definition yields a function $\tilde f_n$  satisfying D1 on $\p\widetilde\Omega_n\setminus \partial \Omega_{n}$. We conclude that $\widetilde\Omega_n$ and $\tilde f_n$ satisfy the assumptions of Theorem \ref{Main2}, yielding a solution $w_n$ to (\ref{lg}) on $\tilde \Omega_n$ with  trace $\tilde f_n$.

We claim that 
due to the continuity of data on $L_n$, solution $w_n$ is continuous at points belonging to $L_n$. 
Indeed, let us suppose that $\tilde\Omega_n\ni \xi_k \to \xi\in L_n$ as $k\to \infty$. By the choice of $\tilde f_n$ we see that $w_n= f(x_n)$ on $\tilde\Omega_n\setminus \Omega_n$. Hence, it is sufficient to consider $\xi_k \in \Omega_n$. For sufficiently large $k$ points $\xi_k$ belong to $\Omega_n\setminus \Omega_{n-1}$. Hence, $w_n(\xi_k)\ge f(x_{n-1})$. This implies that 
$$
\d\{ x: \ w_n(x)\ge w_n(\xi_k)\} \subset \Omega_n\setminus \Omega_{n-1}.
$$
We set 
$$\{\eta^k_1, \eta^k_2\} = \d\{ x: \ w_n(x)\ge w_n(\xi_k)\} \cap 
\d \Omega_{n}.
$$
Of course $\eta^k_i \to \eta_i$, $i=1,2$, when $k\to \infty.$ We claim that $\{\eta_1, \eta_2\}$ intersects $\{x_n, y_n\}$, otherwise $\dist ([\eta_1, \eta_2], L_n)>0$, but this contradicts convergence of $\xi_k$ to $\xi\in L_n$. 

Since $\{\eta_1, \eta_2\}\cap\{x_n, y_n\}\neq\emptyset$, then $\lim_{k\to\infty} w_n (\xi_k) = f(x_n)$, as desired.
As a result, if we set $u_n = w_n \chi_{\Omega_n} + f(x_n) \chi_{\Omega \setminus \Omega_n}$, then
\begin{equation}\label{n1}
    \int_\Omega |D u_n| = \int_{\Omega_n} |Du_n| 
    =\int_{\Omega_n} |D w_n|.
\end{equation}
{\it Step 3.} We claim that $u_n$ are uniformly bounded in $BV$. For this purpose we notice that
$$
\| w_n \|_{L^\infty} = \| u_n \|_{L^\infty} .
$$
Since $w_n$ is a solution to (\ref{lg}) in $\tilde\Omega_n$ with datum $\tilde f_n$, we may now recall (\ref{rtw37}). This yields
\begin{equation}\label{n2}
  \|w_n\|_{L^\infty} \le \| \tilde f_n\|_{L^\infty}\le
  \|  f\|_{L^\infty}
\end{equation}
and
$$
\int_{\Omega_n} |Dw_n| \le TV( \tilde f_n) \diam \tilde\Omega_n
\le TV( f) \diam \tilde\Omega_1,
$$
where the estimate $TV(\tilde f_n )\le TV (f)$ follows just from the definition of $\tilde f_n$. Finally,
\begin{equation}\label{n7}
\| u_n \|_{L^\infty}\le  \|  f\|_{L^\infty}\qquad\hbox{and}\qquad
\int_\Omega |D u_n| \le TV( f) \diam \tilde\Omega_n \le \tilde M.
\end{equation}
This estimate implies that we extract a subsequence (not relabeled) $u_n$ 
converging to $u$, an element of $BV(\Omega).$ Moreover, the lower semicontinuity of the total variation implies that
\begin{equation}\label{n2}
M= \varliminf_{n\to\infty} \int_\Omega |D u_n| \ge \int_\Omega |Du|. 
\end{equation}

{\it Step 4.} We claim that the convergence of $u_n$ to $u$ is strict. For this purpose we will establish that the sequence $|Du_n|(\Omega)$ is increasing. Since we have already observed in Step 2 that $u_{n+1}$ is continuous at points on $L_n$, then we see
$$
\int_\Omega  |Du_{n+1}| = \int_{\Omega\setminus\Omega_{n}} |Du_{n+1}| +
\int_{\Omega_{n}} |Du_{n+1}|.
$$
Also in Step 2. we noticed that
$$
\int_{\Omega_{n}} |Du_{n+1}| = \int_{\Omega_{n}} |Du_{n}|,
$$
hence the 
sequence $| D u_n|(\Omega)$ is increasing. 

We can compare $\int_{\Omega_n}  |Du_n|$ and $\int_{\Omega_n}  |Du|$. The continuity of $u_n$ across $L_n$ implies
$$
\int_\Omega |D u_n| = \int_{\Omega_n}  |Du_{n}|
$$
Since $u_n$ is a solution to the least gradient problem (\ref{lg}) in $\Omega_n$ and with the datum $f_n$, then we have
$$
\int_{\Omega_n}  |Du_{n}|\le \int_{\Omega_n}  |Du| \equiv |Du|(\Omega_n).
$$
Moreover, since the sequence of sets $\{\Omega_n\}_{n=1}^\infty$ is increasing and $|Du|$ is a Radon measure, then we see
$$
\lim_{n\to \infty} |Du| (\Omega_{n}) = |Du| (\Omega).
$$
Combining this with (\ref{n2}), we conclude,
$$
M \ge \lim_{n\to \infty} |Du| (\Omega_{n}) \ge \lim_{n\to \infty} \int_\Omega |Du_{n}| = M.
$$
This means that $u_n$ converges to $u$ in the strict sense.

{\it Step 5.} Now, we could use the continuity of the trace with respect to the strict convergence to claim that $T u = f$.  Alternatively, if $x_0\in \d\Omega$ and $B(x_0,\epsilon)$ is any ball not containing $p_0$, then for sufficiently large $k$ we have
$$
Tu |_{\d\Omega\cap B(x_0,\rho)} = Tu|_{\d\Omega_k\cap B(x_0,\rho)} = 
f_k |_{\d\Omega\cap B(x_0,\rho)} = f |_{\d\Omega\cap B(x_0,\rho)}.
$$
Hence, $T u (x) = f(x)$ for $\cH^1$-a.e. $x\in \d\Omega$.

{\it Step 6.} We claim that $u$ is a least gradient function. Indeed, since  $u$ is an $L^1$ limit of least gradient functions in $\Omega$, then by Miranda Theorem, see \cite{miranda}, $u$ is a least gradient function. Combining this with Step 5, we see that $u$
is a solution to (\ref{lg}).

Finally, we establish estimates on the solution we constructed. Since $u_n $ converges to $u$ a.e., then (\ref{n7}) implies that
$$
\| u \|_{L^\infty} \le \| f \|_{L^\infty} .
$$
Moreover, since $\diam \tilde \Omega_n$ converges to $\diam \Omega,$ then (\ref{n7}) and (\ref{n2}) imply (\ref{rtw310}).
\end{proof}

 \subsection{Discontinuous data with infinitely many humps}
 We assume in this section that the data  have infinitely many humps on sides of $\p \Omega$. To be specific, we assume $\ell=[p_l,p_r]$ is a side of $\p\Omega$, 
 where $f$ has infinitely many humps $I_i$, $i\in\cI$, then since $f$ satisfies the admissibility condition D2 the lengths of $I_i$ must converge to 0 and the humps endpoints must converge to the one of the endpoints of $\ell$. For the sake of 
 definiteness, we assume that $p_l$ is the point of accumulation of $I_i$'s endpoints.

When we deal with an infinite number of humps, then we assume that they accumulate at a single point $p_0$.  Moreover, we assume that a side of $\d\Omega$ contains an infinite number of them and $p_0$ is its endpoint. 
This restriction is introduced solely for the sake of simplicity of the exposition. The same argument works if we have a finite number of accumulation points.

The above assumptions imply that either another facet contains infinitely many humps or  an infinite number of sides contains a finite number of humps.
Here is our main observation.

\begin{theorem}\label{Main3}
Let us suppose that  $\Omega$ is a convex polygonal domain. Function $f$ is in $BV(\partial\Omega)$, it satisfies conditions D1 or D2, D3, the Order Preserving and Data Consistency conditions. There is a side $\ell$ containing an infinite number of humps accumulating at $p_0$, an end point of $\ell$.
Then, there exists a solution $u$ to the least gradient problem (\ref{lg}). Moreover,
\begin{equation}\label{rtw312}
\| u \|_{L^\infty} \le \| f  \|_{L^\infty}, \qquad
\int_\Omega |Du| \le \diam \Omega\, TV (f).    
\end{equation}
 \end{theorem}

\begin{proof}
{\it Step 1.} We proceed by constructing  a 
sequence of least gradient functions on an increasing sequence of subsets of $\Omega$. The approximating sets $\Omega_n$ are constructed exactly as in \cite[Section 3.3]{RySa}. By Theorem  \ref{Main2}, 
we will infer existence of solutions to the least gradient problems, $u_n$, on $\Omega_n$. They will be used to find a solution to the problem in question.

Let us change slightly the notation and we shall write $[p_0, p_1]$ for $\ell$.
We assume that this side has an infinite number of humps, $\{ \bar I_n\}_{n=1}^\infty$, $\bar I_n =[a_n, b_n]$, 
accumulating at  endpoint  $p_0$. For the sake of definiteness we also assume that $|p_0-a_n| < |p_0 - b_n|$ for all $n\in \bN$.
We choose $y_n$ and $z_n$ as in (\ref{ddfyz}), i.e.
$$
\dist(a_n, \d \Omega\setminus I_n) = |a_n - y_n |,\qquad
\dist(b_n, \d \Omega\setminus I_n) = |b_n - z_n |,\qquad y_n, z_n \notin [p_0, p_1].
$$
We set
$L_n = [a_n, y_n]$ and
$\Omega_n = \Omega \setminus H(L_n, p_0)$, where 
$H(L_n, p_0)$ was defined in Step 1 of the proof of Theorem \ref{td-nsk}. We also set
$T_n = \Omega \setminus \Omega_n$. 

We define the boundary data on $\p\Omega_n$,
 $$
 f_n(x)=\begin{cases} f(x) \qquad &x \in \partial \Omega_n\cap \partial \Omega,\\
 e_n \qquad &x\in L_n.
 \end{cases}
 $$
 Here, $e_n$ is the value of $f$ on the hump $I_n$. Since $p_0$ is the only  possible accumulation point of the sides of $\p \Omega$, we see that $\Omega_n$ has a finite number of sides and $f_n$ has a finite number of humps.

{\it Step 2.} We will check that  $f_n$ satisfies the admissibility condition D1 or D2 on $\d \Omega_n$.  
First, we notice that we did not change the 
hump structure on $\ell\cap \p  \Omega_n \setminus I_n$. 
As a result, D2 holds on $\ell\cap \p  \Omega_n$. 

Moreover, $f_n$ on interval $L_n$ is constant, i.e.  monotone. We have to show that condition D1 is satisfied on $L_n$. Indeed, $f_n$ is constant on $I_n$ and on $L_n$, hence it is  continuous at $a_n$. We have to  look at the behavior of $f$ near $y_n$. If  $f$ is continuous at this point, then D1 holds. 

Now, we have to consider $y_n$ as a discontinuity point of $f$ and $f_n$. We have two possibilities for $y_n$: it is either in the relative interior of a side $\ell'_n$ or it is a corner, i.e. there are sides $\ell'_n, \ell_n''$ such that $y_n \in \ell_n'\cap\ell_n''$. In the first case, it follows from D3 that  there must be an interval $y_n\in J_n\subset \ell_n'$, such that $f$ restricted to $J_n$ is monotone.
If so, then the DCC implies that $f_n$ restricted to $I_n \cup L_n\cup J_n$ must be monotone, thus $f_n$ satisfies D1 on $L_n$.

If $y_n$ is a corner, then due to Corollary \ref{cr1} $f$ is continuous at $y_n$. Thus, condition D1 holds on $L_n$. 

We must check that $f_n$ satisfies D1 or D2 on a side $\ell'\subset \d\Omega_n$ containing $y_n$. We may begin our analysis by assuming that $f$ on 
$\ell'$ satisfies D1. In this case, the DCC implies that $f_n$ on $\ell'\cup L_n\cup I_n$ is monotone, so D1 holds for $f_n$ on $\ell'$.

Now, we look at the case, when D2 was true for $f$ on a  side of $\p\Omega$, $\tilde\ell$ containing $\ell_n'$. Then, the structure of D2 implies existence of an interval $J_n\subset \ell_n'$ containing $y_n$ such that $f|_{J_n}$ is monotone. Now, the DCC on $I_n$ implies that if $f$ attains a local maximum on $I_n$, then $f|_{J_n}\ge e_n$ or if  $f$ attains a local minimum on $I_n$, then $f|_{J_n}\le e_n$. Hence, 
$f_n$ is monotone on $J\cup L_n$ and D2 holds on $\ell_n'$. 

{\it Step 3.} We will construct a sequence of approximating solutions. We have already noticed that $\Omega_n$ and $f_n$ satisfy the assumption of our basic existence result, Theorem \ref{Main2}. Thus, there exists a solution $v_n$ of problem (\ref{lg}) on $\Omega_n$ with datum $f_n$. However, this is not enough information for us. Let us call by $L(\ell'_n)$ the line containing $\ell_n'$ and $L(\ell)$ the line containing $\ell$. The lines $L(\ell'_n)$ and $L(\ell)$ intersect at $\bar p_n$. We define $T_n = \conv(\bar p_n, L_n)$ and we set $\tilde\Omega_n = \Omega_n\cup T_n$. Of course $\tilde\Omega_n $ is convex and it contains $\Omega$. We define $\bar f_n \in BV(\d \tilde\Omega_n )$ by the following formula,
$$
\bar f_n(x) = \left\{
\begin{array}{ll}
f_n(x) & x\in \tilde\Omega_n\cap \Omega_n,\\
e_n & x\in \tilde\Omega_n\setminus \bar\Omega_n .
\end{array}\right.
$$
We claim that each $\bar f_n$ satisfies D1 or D2 on the sides of $\tilde\Omega_n$. Indeed:\\
(1) On $\ell$ functions $f$ and $f_n$ have humps. By definition, any hump is separated from the endpoint of $\ell$ by an interval, where $f$ is monotone. The definition of $\bar f_n$ extends this interval of monotonicity. So,  $\bar f_n$ satisfies D2 on $\ell$.\\
(2) We have already noticed that $f_n$ is monotone on $J_n$, so is $\bar f_n$ on $J \cup [\bar p_n, y_n]$.

{\it Step 4.} 
Since $\tilde\Omega_n$ has a finite number of sides, $\bar f_n$ has a finite number of humps and it satisfies the remaining assumptions of Theorem \ref{Main2}, we may invoke this result. It guarantees existence of a solution $w_n$ to (\ref{lg}) on $\tilde\Omega_n$ with datum $\bar f_n$.

We claim that the trace of $w_{n+1}|_{\Omega_n}$ is equal to $f_n$. In fact, it suffices to check that $T(w_{n+1}|_{\Omega_n}) = e_n$ on $L_n$. Since due to Corollary \ref{c39} $w_{n+1} = e_n$ on $\conv(a_n,b_n, y_n,z_n)$ our claim follows.

{\it Step 5.} 
We set
$$
u_n = w_n|_\Omega.
$$
We notice that $u_n(x) = e_n$ for $x\in T_n$ and $u_n$, $w_n$ are continuous across $L_n$.  The continuity of $u_n$ is established as in the course of proof of Theorem \ref{td-nsk}. Thus, we notice that
\begin{equation}\label{19r1}
    \int_\Omega |Du_n| = \int_{\tilde\Omega_{n}} |Dw_n| = \int_{\Omega_n} |Dw_n|.
\end{equation}
It is easy to see that
$$
\int_\Omega  |Du_{n+1}| = \int_{\Omega_{n+1}} |D w_{n+1}| \ge 
\int_{\Omega_{n}} |Dw_{n+1}| \ge \int_{\Omega_{n}} |Dw_{n}| = \int_{\Omega} |Du_n| ,
$$
where we took into account Step 4.
Thus, the sequence $| D u_n|(\Omega)$ is increasing. The same methods we used in the course of  Theorem \ref{td-nsk} yield
$$
\sup_{n\in \bN} | D u_n|(\Omega) = M< \infty.
$$
Then, we can select a subsequence $\{u_{n_k}\}_{k=1}^\infty$ convergent in $L^1$ to $u$. By the lower semicontinuity of the total variation, we deduce
$$
\varliminf_{k\to\infty} \int_\Omega  |Du_{n}| \ge \int_\Omega  |Du|.
$$
On the other hand, we can compare $\int_{\Omega_n}  |Du_n|$ and $\int_{\Omega_n}  |Du|$. Now, (\ref{19r1}) and the continuity of $u_n$ across $L_n$ imply
$$
\int_\Omega |D u_n| = \int_{\Omega_n}  |Du_{n}| = \int_{\Omega_n}  |Dw_{n}| 
\le \int_{\Omega_n}  |Du| \equiv |Du|(\Omega_n).
$$
Since the sequence of sets $\{\Omega_n\}_{n=1}^\infty$ is increasing and $|Du|$ is a Radon measure, then we see
$$
\lim_{n\to \infty} |Du| (\Omega_{n}) = |Du| (\Omega).
$$
Then, we conclude,
$$
M \ge \lim_{n\to \infty} |Du| (\Omega_{n}) \ge \lim_{n\to \infty} \int_\Omega |Du_{n}| = M.
$$
This means that $u_n$ converges to $u$ in the strict sense. Then we have convergence of traces,
$$
T u_n \to Tu
$$
when $n\to\infty.$ Thus, $Tu = f.$

{\it Step 6.} Since functions $w_n$ were of least gradients, so were their restrictions to $\Omega$, i.e. $u_n$. Hence, by Miranda theorem \cite{miranda},  $u$, the limit of the sequence $u_n$ is a least gradient function. Since by the previous step it has the correct trace, we see
that $u$ is a desired solution.

Finally, due to (\ref{19r1}) and the lower semicontinuity of the variation we have,
$$
|Du|(\Omega) \le \liminf_{n\to\infty} |Du_n|(\Omega)
\le \liminf_{n\to\infty} |Dw_n|(\tilde\Omega_n)\le TV(\tilde f_n)\,\diam \tilde\Omega_n,
$$
where the last inequality follows from Theorem \ref{Main2}. Since 
$\lim_{n\to\infty} \diam(\tilde\Omega_n) = \diam \Omega$ and $TV(\tilde f_n) = TV(f)$, we deduce that
$$
\int_\Omega |Du| \le \diam \Omega \, TV(f).
$$
It is easy to see that our construction also  yields the first part of (\ref{rtw312}).
\end{proof}

\section*{Acknowledgement} The work of the authors was in part
supported by the Research Grant 2015/19/P/ST1/02618 financed by the
National Science Centre, Poland, entitled: Variational Problems in
Optical Engineering and Free Material Design.

PR was in part supported by  the Research Grant no 2017/26/M/ST1/00700 
financed by the National Science Centre, Poland. 

The work of AS was in part performed at the University of Warsaw.



\vspace{-.2cm}
\hspace{3cm}
\begin{wrapfigure}{l}{0.15\textwidth}
\includegraphics[width=2.5 cm, height= 2cm]{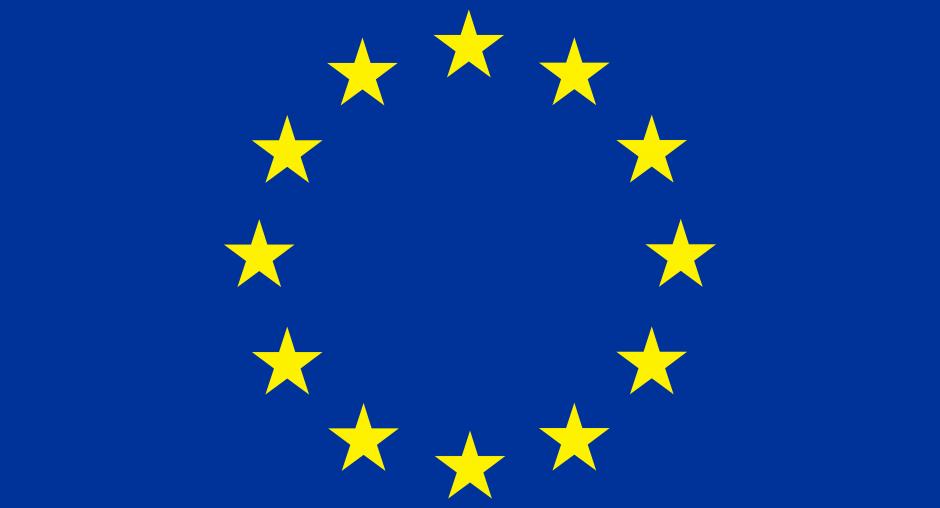}  
\end{wrapfigure}\\
\\{\footnotesize This project has received funding from the European Union's Horizon 2020 research and innovation program under the Marie Sk\l{}odowska-Curie grant agreement No 665778.}
\newline

 \hspace{-5cm}
 
  \hspace{5cm}

\end{document}